\documentclass[reqno,10pt,a4paper]{amsart}
\usepackage{amstext}
\usepackage{amsthm}
\usepackage{amsopn}
\usepackage{amsfonts}
\usepackage{amsmath}
\usepackage{amssymb}
\usepackage{amsthm}
\usepackage{amscd}
\usepackage{enumerate}
\usepackage{color}
\usepackage{epsfig}
\usepackage{graphicx}
\def\nc{\newcommand}
\newcommand{\N}{\mathbb{N}}

\newcommand{\R}{\mathbb{R}}
\newcommand{\C}{\mathbb{C}}
\newcommand{\lam}{\lambda}
\newcommand{\eps}{\epsilon}
\newcommand{\Zp}{\mathcal Z}

\newcommand{\calA}{\mathcal A}

\newcommand{\all}[2]{ \left \{\, {#1} \, : \, {#2} \, \right \} }
\newcommand{\set}[1]{{\{#1\}}}
\newcommand{\cl}{\rm cl}
\newcommand{\Ho}{{\mathcal H}}
\newcommand{\htop}{h_{top}}
\newcommand{\Hol}{{\rm H\ddot{o}l}}

\DeclareMathOperator*{\var}{var}

\newcommand{\abs}[1]{\left | {#1}\right | }
\newcommand{\norm}[2]{\left\| {#1}\right\| _{#2}}

\newcommand{\tnorm}[1]{%
  {\left\vert\kern-0.25ex\left\vert\kern-0.25ex\left\vert {#1} 
   \right\vert\kern-0.25ex\right\vert\kern-0.25ex\right\vert}}
\newcommand{\wnorm}[1]{\abs{#1}_w}
\newcommand{\esup}[2]{\abs{#1}_{#2,\infty}}
\newcommand{\dist}{{\rm dist}}
\newcommand{\ncircle}[1]{\text{\raisebox{.5pt}{\textcircled{\raisebox{-.9pt} {#1}}}}}

\newcommand{\tr}{\mathcal L}

\theoremstyle{plain}
\newtheorem{theorem}{Theorem}[section]
\newtheorem{proposition}[theorem]{Proposition}
\newtheorem{lemma}[theorem]{Lemma}
\newtheorem{corollary}[theorem]{Corollary}

  \newtheoremstyle{TheoremNum}
        {}             %
        {}              %
        {\itshape}                      %
        {}                              %
        {\bfseries}                     %
        {.}                             %
        { }                             %
        {\thmname{#1}\thmnote{ \bfseries #3}}%
   \theoremstyle{TheoremNum}

\theoremstyle{definition}
\newtheorem{definition}[theorem]{Definition}
\newtheorem{remark}[theorem]{Remark}

\newtheoremstyle{remarkbreak}%
  {}%
  {}%
  {\upshape}%
  {}%
  {\bfseries}%
  {.}%
  {\newline}%
  {}%

\theoremstyle{remarkbreak}

\newcommand{\rr}{\rightarrow}

\def\mystrut  {\rule[0ex]{0ex}{1.55ex}{}}
\newcommand{\myvar}[1]{\var_{\mystrut #1}}
\def\ds{\displaystyle}

\DeclareMathOperator*{\Card}{card}

\def\calD {{\mathcal D}}
\def\calG {{\mathcal G}}
\def\hatcalD {\widehat{\mathcal D}}

\nc{\dual}[1]{\langle #1 \rangle}

\def\rr{\rightarrow}
\def\bfone{{\mathbf 1}}
\def\htop {h_{\rm top}}
\nc{\Tr}{\mbox{\rm Tr}}
\nc{\mytrace}[1]{\Tr \left( \sstrut #1 \right)}

\nc{\mat}[2]  {\left(  \! \begin{array}{#1} #2 \end{array}\! \right)}

\begin{document}

\title {Entropy-continuity for interval maps with holes}

\author[O.F.~Bandtlow]{Oscar F.~Bandtlow}

\address{Oscar F.~Bandtlow\\
School of Mathematical Sciences\\
Queen Mary University of London\\
London E3 4NS, UK\\
Email: o.bandtlow@qmul.ac.uk}

\author[H.H. Rugh]{Hans Henrik Rugh}
\address{Hans Henrik Rugh\\
Laboratoire de Math\'ematiques d'Orsay, 
Universit\'e Paris-Sud, CNRS, 
Universit\'e Paris-Saclay\\
91405 Orsay Cedex, France.
Email:  hans-henrik.rugh@math.u-psud.fr}

\date{July 31, 2016}

\begin{abstract} 
We study the dependence of the topological entropy of 
piecewise monotonic maps with holes under 
perturbations, for example 
sliding a hole of fixed size at uniform speed or expanding a hole at a uniform rate. 
We show that under suitable conditions the 
topological entropy varies locally H\"older continuously with the 
local H\"older 
exponent depending itself on the 
value of the topological entropy. 
\end{abstract}

\maketitle

\section{Introduction and main result}

Given a piecewise monotonic map $T$ of an interval $I$, one measure
of the `complexity' of the map is the topological entropy
$\htop(T)$, given by the exponential growth rate of
the number of monotonicity intervals of iterates of the map.
It turns out that varying the map may produce discontinuities of the topological entropy.
As a simple, yet illustrative example, consider  for $0\leq a \leq 1$,
the scaled Farey map $T_a:[0,1]\rr [0,1]$ given by 
\[ T_a(x)=
 \begin{cases}
 a\frac{x}{1-x} & \text{for $x\in [0,1/2]$},\\
 a\frac{1-x}{x} & \text{for $x\in (1/2,1]$}.
 \end{cases}
 \]

It is not difficult to see that if 
$0\leq a \leq \frac12$ then 
$\htop(T_a)=0$. For $\frac12 < a \leq 1$ the map has a closed invariant set on which it is topologically conjugate to the tent map, so  $\htop(T_a)=\log 2$. Thus, 
as $a$ passes a half, the sudden appearance of a full tent map produces a 
discontinuity of the topological entropy. On the other hand, it is known that for certain one-parameter families of smooth maps the topological entropy is even H\"older continuous (see, for example, \cite{guckenheimer}). 

A somewhat milder way of changing the map is to fix the map itself
but to change the interval of definition, that is to introduce `holes'
in the map. In dynamical terms, if a point in the domain of definition hits the hole under iteration of the map, it is not 
allowed to evolve any further. The motivational idea is a physical system for which mass leaks through an 
opening in phase space. 

This set-up has received considerable attention over the past decades (see \cite{BBF} 
for a snapshot of recent developments). Most of the efforts have been directed to 
studying the escape rate of mass through the hole as well as 
showing the existence of conditionally invariant measures with good statistical properties for various classes of maps, 
as well as conditions on the hole, typically that the hole is either 
part of a Markov partition for the map (see, for example, \cite{PY,CMS,CMa}) 
or that the hole is `small' (see, for example, \cite{CMaT,LM,BC,BrDM,DemWY,DW}). 
In a similar vein, the behaviour of the escape rate as the size of the hole shrinks to zero has also been studied in some detail (see, for example, 
\cite{BY, KL09, FP, dettmann, CKD}). 

In this article, we shall focus on the regularity of the topological entropy $\htop(T,H)$ 
of a fixed piecewise monotonic expanding interval map $T:I\to I$ as a function of the hole $H$, without any assumptions on the size or position of $H$. 
Roughly speaking we shall show that if $\mathcal S$ is a neighbourhood of $t\in \R$ and 
$(H_s)_{s\in \mathcal S}$ is a family of holes varying sufficiently regularly in the space of holes, then,  
under a mild expansivity condition on $T$, the function $s\mapsto \htop(T,H_s)$ is continuous at $t$. 
Moreover, if $\htop(T,H_t)>0$, then $s\mapsto \htop(T,H_s)$ is even H\"older continuous at $t$ 
with H\"older exponent at $t$ depending on $\htop(T,H_t)$. 

In order to provide a more precise formulation of these results, we require some more notation. 
\begin{definition} \label{def:map}
Let $I$ be a compact interval and let $\Zp$ be a finite collection of disjoint open subintervals of $I$.
We say that
$T:\bigcup \Zp \to I$ is a  \emph{piecewise monotonic $C^1$-map} with  \emph{initial partition $\Zp$} if 
for each  $Z\in \Zp$ the restriction $T_{|Z}$ of $T$ to $Z$ is strictly monotonic and extends as 
a $C^1$-function to the closure $\cl(Z)$ of $Z$.
\end{definition} 

Suppose now that we are given a piecewise monotonic $C^1$-map $T$ with initial partition $\Zp$.
For $n\in \N$ we write
\[  \Zp_n=\all{Z_0\cap T^{-1} Z_1\cap \ldots \cap T^{-(n-1)}Z_{n-1}\neq \emptyset}{Z_i\in \Zp \text{ for all $i$}}
\]
for the collection of \emph{cylinder sets of level $n$} and observe 
that $T^n$ is well-defined and 
strictly monotonic on each element of $\Zp_n$. In order to avoid trivialities we shall assume 
throughout this article that  $\Zp_n$ is non-empty for all $n$.

Let $D(T^n)$ denote the derivative of $T^n$ and 
$m$ Lebesgue measure on $I$. We set
\[ \Xi_n(T)=\frac{1}{n}\sup_{Z\in \Zp_n} \log \frac{\sup_{x\in Z}\abs{D(T^n)(x)}}{ \makebox{$m(T^nZ)$}\,,   }
\]
 and define
\[ \Xi(T)=\limsup_{n\to \infty} \Xi_n(T)\,.\]
The number $\Xi(T)$ quantifies expansivity of the map. 
In the particular case where $T$ is Markov\footnote{A piecewise monotonic $C^1$-map $T:\bigcup \Zp \to I$ 
is said to be \emph{Markov} if for every $Z\in \Zp$  
the closure of its image $\cl(T(Z))$ is a union of 
closures of elements of the initial partition $\Zp$.},
it follows that 
$m(T^n Z)$ is uniformly bounded from below, so the number $\Xi(T)$ measures the growth rate of the maximal 
expansion of $T^n$. 

Let us now make the notion of a hole more precise. For $N\in \N$ 
we write
$\Ho_N(I)$ for the collection of sets $H\subset I$ which may be written as a union of
at most $N$ closed subintervals of $I$. We then define
 \begin{equation}
  \label{def:holespace}
    \Ho(I)=\bigcup_{N\in \N}\Ho_N(I)
  \end{equation}
to be our `space of holes' and equip it with the following distance function 
 \begin{equation}
   \label{def:holedist}
    \dist(H_1,H_2)=m((H_1\setminus H_2) \cup (H_2\setminus H_1)) \quad (\forall H_1,H_2 \in \Ho(I)) \, . 
  \end{equation}
Note that this distance is a pseudometric, that is, it is a \textit{bona fide} metric, 
except that the distance between two holes may vanish without the holes coinciding. 
Note also that it does not distinguish the number of intervals 
constituting a hole. 

Given a map $T:\bigcup \Zp\to I$ and a hole $H$ as above, 
we will be interested in the dynamics of the map restricted to the complement of $H$. More precisely, we consider $T$ on the initial partition
\begin{equation*}
 \Zp^H=\all{Z\cap (I\setminus H) \neq \emptyset}{Z\in \Zp}
 \label{def:ZH} 
  \end{equation*}
and let $ \Zp^H_n$ denote the corresponding collection of cylinder sets
of level $n$. In passing, we note that elements of  $\Zp^H_n$ are not necessarily intervals in the usual sense.  
The (non-negative)\footnote{
There is no standard convention for the case when all points escape, that is, when 
$\Zp^H_n$ becomes eventually trivial. Here, we choose to
set $\htop(T,H)=0$ in this case, as this assures continuity of
the topological entropy, as we shall see in the following.}
topological entropy
 of the restricted map is then given by the exponential growth
 rate of the number of elements in $\Zp^H_n$ 
\begin{equation*}
   \htop(T, H) = \lim_{n \rightarrow \infty}
      \frac1n \log^+ \Card \Zp^H_n \geq 0 ,
      \label{def:htop}
\end{equation*}
 where $\log^+(t) = \max\{0,\log(t)\}$.
  By \cite{MS}, this coincides with 
other definitions of the topological entropy. 

Associated with the restricted map we also have a transfer operator acting on $BV(I)$, the space of functions of bounded variation on $I$, 
\[ (\tr_{I\setminus H}f)(x) = \sum_{y\in T^{-1}(x)}\chi_{I\setminus H}(y)f(y) \,.\]
It turns out that if $\htop(T,H)>0$ then $\exp(\htop(T,H))$ is an isolated eigenvalue (in fact the largest in modulus) of $\tr_{I\setminus H}$, and hence also a pole of the resolvent of $\tr_{I\setminus H}$. As we shall see shortly, the order of this pole plays a role for the degree of regularity of the topological entropy of $T$ at $H$.

Our main result on the regularity of the topological entropy as a function of the hole, to be proved in 
Section~\ref{sec:htop}, can now be formulated as follows. 
 
\begin{theorem} \label{theorem:Main}
Suppose that $T:\bigcup \Zp \to I$ is a 
piecewise  monotonic $C^1$-map with $0 < \Xi(T)<+\infty$.
Let $\mathcal S$ be a neighbourhood of $t\in \R$ and suppose that  
$(H_s)_{s\in \mathcal{S}} \subset \Ho(I)$ is a family of holes which is
Lipschitz continuous in $s$ and for which the number of 
holes is uniformly bounded.
Then $s\mapsto \htop(T, H_{s})$ is continuous at $t$.

Furthermore, if \  $\htop(T, H_{t})>0$  \ 
then $s\mapsto \htop(T, H_{s})$ 
is H\"older continuous at $t$ with local H\"older exponent
(see Definition \ref{def:Holder}) 
satisfying
\begin{equation}
 \label{eq:HolExp}
  \Hol(\htop(T,H_{\cdot}), t)\geq 
      \frac{\htop(T,H_{t}) }{ p \; \Xi(T)}\,, 
\end{equation}
where $p$ is the order of the pole of the resolvent 
of $\tr_{\chi_{I\setminus H_t}}$ at $\exp(\htop(T,H_t))$. 
\end{theorem}

It is rather curious that the local H\"older exponent of the topological entropy at a particular hole depends 
itself on the value of the topological entropy at that hole. A similar behaviour of the topological entropy, 
albeit not for maps with holes, but for one-parameter deformations of the tent map, 
has already been conjectured by Isola and Politi in the early nineties (see ~\cite{IP}). 
Our numerical simulations also indicate that the factor $p$ can not be omitted in the above formula.  

The theorem above is in fact a special case of a more general theorem with similar hypotheses and similar 
conclusions, but with the behaviour of the topological entropy as a function of the hole replaced by 
the behaviour of the topological pressure as a function of the potential (see Corollary~\ref{coro:main}). 

The proof of the theorem above and its generalisation relies on a 
spectral perturbation theorem of 
Keller and Liverani~\cite{KL99}, which has opened up a rich seam of applications in various areas. 
In fact, our proof relies on a refinement of the Keller-Liverani theorem 
(see Corollary~\ref{coro:ourcoro}) which is of interest in its own right and elucidates the role of the pole of the resolvent of the perturbed operator. 

As an illustration of the theorem above we consider the topological entropy of the 
doubling map $T(x)=2x \mod 1$ on the unit interval with a uniformly left-expanding hole, that is, we are 
interested in the regularity of the function 
\[ (1/2,1) \ni a \mapsto h(a)=\htop(T, [a,1])\,.\] 
It is not difficult to see that $\Xi(T)=\log 2$, so the first part of the 
theorem above immediately implies  that $h$ is continuous, thus giving a new proof of a result 
originally due to Urba\'{n}ski \cite[Theorem~1]{U}. Moreover, as we shall show in 
Section~\ref{sec:ex},  the order of the pole of $\tr_{I\setminus [a,1]}$ at $\exp(h(a))$ is one for each 
$a\in (1/2,1)$, 
so $h$ is H\"older continuous at each $a\in (1/2,1)$ with local H\"older exponent satisfying 
\begin{equation}
\label{eq:CT}
 \Hol(h,a)\geq \frac{h(a)}{\log 2}\,.
 \end{equation}
The above lower bound for the local H\"older exponent 
was recently obtained by Carminati and Tiozzo \cite{CT}, using a different 
approach of a combinatorial flavour, which, as a bonus, also yields that there is equality in (\ref{eq:CT}) for all 
$a\in (1/2,1)$ at which $h$ is not locally constant. 
In fact, they are able to show that the local  H\"older exponent of the topological entropy of 
the $d$-adic map $T(x)=dx \mod 1$ with $d\in \N\setminus \set{1}$ with uniformly left-expanding hole equals the topological entropy divided by $\log d$ for all points at which the topological entropy is not locally constant. 

 This paper is organised as follows. In Section~\ref{sec:setup} we consider 
piecewise monotonic interval maps and study the corresponding transfer operators with general weights on 
spaces of functions of bounded variation. The main goal will be to obtain a Lasota-Yorke inequality for the 
transfer operator (see Proposition~\ref{prop:LYI}), which is sufficiently uniform in order for the Keller-Liverani 
Theorem to apply. Section~\ref{sec:KL} is devoted to proving our refinement of the Keller-Liverani Theorem 
(see Corollary~\ref{coro:ourcoro}). In the following Section~\ref{sec:holder}, we define the notion of the 
pressure $P(T,g)$ of a piecewise monotonic interval map $T$ with a given weight $g$ 
in a form suitable for our applications and show its equivalence to other definitions 
(see Theorem~\ref{thm:bkpressure}). 
We then go on to prove our main result on the regularity of $P(T,g)$ as a function of $g$ (see Corollary~
\ref{coro:main}). In Section~\ref{sec:htop} we specialise the results of the previous section to prove 
our Theorem~\ref{theorem:Main} on the regularity of the topological entropy as a function of the hole. In
Section~\ref{sec:ex} we apply the results of the previous section to the doubling map with 
left-expanding hole, for which we are able to show that the order of the pole is always equal to one. 
We also consider the doubling map with a sliding hole of fixed size, showing that a double pole can 
suddenly occur. 

\section{Setup}
\label{sec:setup}

Let $T:\bigcup \Zp \rightarrow I$ be a
piecewise monotonic $C^1$-map.
Given a bounded function $g:\bigcup \Zp \to \C$ we define the 
\emph{Ruelle transfer operator of $T$ with weight $g$} by 
\[ 
  \tr_gf=\sum_{Z\in \Zp} (f\cdot g)\circ T_{|Z}^{-1} \cdot \chi_{TZ}\,, 
\]
for any bounded $f:I\to \C$. Here, $T_{|Z}^{-1}$ denotes the inverse
of the 
restriction of $T$ to $Z$. 
It is not difficult to see that the $n$-th power of $\tr_g$ can be written 
\[ 
  \tr_g^nf=\sum_{Z\in \Zp_n} (f\cdot g_n)\circ T_{|Z}^{-n} \cdot \chi_{T^nZ}\, 
\]
where $T_{|Z}^{-n}$ denotes the inverse of $T^n$ restricted to $Z\in \Zp_n$ 
and $g_n:\bigcup \Zp_n \to \C$ is given by 
\[ 
  g_n=\prod_{k=0}^{n-1}g\circ T^k\,.
\]
We shall see presently that, for suitable $g$, 
the transfer operator $\tr_g$ has nice spectral properties on the space $BV$ of 
functions of bounded variation, 
the definition of which we briefly 
recall. 

Let $m$ denote Lebesgue measure on $\R$. For $A\subset \R$ measurable
we write $L^1(A)=L^1(A,m)$ to denote the Banach space of
(equivalence classes) of $m$-integrable functions on $A$ with the
usual norm. Suppose now that $J\subset \R$ 
is a bounded interval. If $f:J\to \C$ is an ordinary function we write 
\[ 
  \bigvee_J f = \sup \all{ \sum_{i=1}^n\abs{f(c_{i+1})-f(c_{i})} }
                            { n\in\N,c_0<c_1<\cdots < c_n, c_i\in J } 
\]
for the total variation of $f$, while for $f\in L^1(J)$ 
the total variation is given by
\[ 
  \var_J(f)=\inf \all{\bigvee_J \tilde{f} }
                         {\text{$\tilde{f}$ is a version of $f$}}\,.
\]
The space of functions of bounded variation is now defined by 
 \[ BV(J)=\all{f\in L^1(J)}{\var_J(f)<\infty}\,. \]
When equipped with the norm 
\[ \norm{f}{BV(J)}=\var_J(f)+\int_J\abs{f}\,dm \]
it becomes a Banach space, which is compactly embedded in $L^1(J)$ 
(see, for example \cite[Lemma 5]{HK} or 
\cite[Theorem~1.19]{giusti}). For other characterisations of this
space, see, for example \cite[pp.~26-9]{giusti} or 
\cite[Theorem~2.3.12]{BG}. 

We briefly mention $BV$ spaces over more general sets, 
which are defined as follows. 
Let $\mathcal J$ be a finite collection of mutually disjoint
open
intervals and $A=\bigcup \mathcal{J}$. For 
$f\in L^1(A,m)$, we let 
\[ \var_{A}(f)=\sum_{J\in \mathcal{J}}\var_{J}(f) \]
and define 
 \[ BV(A)=\all{f\in L^1(A)}{\var_{A}(f)<\infty}\,. \]

It turns out that, for suitable $g$, 
the transfer operator $\tr_g$ of a piecewise monotonic $C^1$ map
$T:\bigcup \Zp \to I$ 
is a continuous endomorphism of $BV(I)$ with 
discrete peripheral spectrum (see \cite{HK, rychlik, Keller, BK}). 
Moreover, 
we shall see that the peripheral spectrum is stable under suitable
perturbations of $g$. 
This will follow from an 
application of the 
Keller-Liverani perturbation theorem \cite{KL99}. In order to apply it
we require a 
Lasota-Yorke inequality for 
$\tr_g$, which is suitably uniform in $g$. 

In our derivation we shall follow the approach in \cite{BK}, where 
a 
$BV$-$L^\infty$ 
Lasota-Yorke inequality is proved. 
For reasons that will become apparent later on, we require a 
$BV$-$L^1$ Lasota-Yorke inequality, similar to the one by Rychlik~\cite{rychlik},  
but with more explicit control of the coefficients. 

We start by quickly summarising some  
useful properties of 
variation. In order to declutter notation it will be useful to
introduce the following shorthand: for $J$ a real interval and 
$f\in L^\infty(J,m)$ we write 
\[ \esup{f}{J}:=\norm{f}{L^\infty(J)} \] 
for its $L^\infty$-norm.   

\begin{lemma} 
\label{lem:varlem}
Let $J$ and $K$ be  bounded non-empty intervals and let $f\in BV(J)$.   
  \begin{enumerate}[(a)]
     \item \label{varlem:a} 
          If $K\subset J$, then 
          \[ \var_K(f)\leq \var_J(f)\,. \]
     \item \label{varlem:b} If $K\subset J$, then 
          \[ \var_{J}(f\chi_{K})\leq \var_K(f)+2\esup{f}{K}\,.\]
     \item \label{varlem:c} If  $S:K\to J$ is a homeomorphism, then 
          \[ \var_K(f\circ S)=\var_J(f)\,.\]
     \item \label{varlem:d} 
          We have  
          \[ \esup{f}{J}\leq \var_J(f)+\frac{1}{m(J)}\int_J\abs{f}\,dm\,.\]
      \item \label{varlem:e} 
          If $f_1,\ldots,f_n\in BV(J)$, then 
          \[ 
            \var_J(\prod_{k=1}^nf_k)\leq \sum_{k=1}^{n} \esup{\prod_{i=1}^{k-1}f_i}{J}\var_J(f_k)
            \esup{\prod_{i=k+1}^{n}f_i}{J}\,.
          \]
       \item \label{varlem:f} 
         If $J_1,\ldots, J_n$ are subintervals of $J$ with mutually disjoint interiors then 
          \[ \sum_{k=1}^n\var_{J_k}f \leq \var_J(f)\,.\]
  \end{enumerate}
\end{lemma}

\begin{proof}
These are all well-known, apart, perhaps, from (e), which is proved as
follows. If $\tilde{f}_j$ is a version of $f_j$ and $x,y\in J$ we
have  
\[  \prod_{k=1}^n \tilde{f}_k (x)  - \prod_{k=1}^n \tilde{f}_k (y)
     =
     \sum_{k=1}^n \prod_{j<k} \tilde{f}_j(y) 
                            (\tilde{f}_k(x) - \tilde{f}_k(y))
                            \prod_{k<l} \tilde{f}_l(x)\,,
\]
so 
\[ 
    \bigvee_J(\prod_{k=1}^n\tilde{f}_k) 
     \leq 
    \sum_{k=1}^{n} \sup_J \left | \prod_{i=1}^{k-1}\tilde{f}_i \right | 
            \bigvee_J(\tilde{f}_k)  \sup_J \left |
              \prod_{i=k+1}^{n}\tilde{f}_i \right | \,,
\]
from which the assertion follows. 
\end{proof}

We shall now derive a number of auxiliary results, including a uniform Lasota-Yorke inequality for the 
action of the transfer operator on functions of bounded variation. 
Here and in the following we use $D(T^n)$ to denote the derivative of $T^n$.

\begin{lemma} \label{lem:varLg} \label{lem:intLg}
Let $T:\bigcup \Zp \to I$ be a piecewise monotonic $C^1$-map and let $g\in BV(\bigcup \Zp)$. 
For $n\in \N$ and $Z\in \Zp_n$ write 
\begin{align}
   a_n(Z)&=3\esup{g_n}{Z}+\var_Z(g_n) \label{eq:anZ}\,, \\
   A'_n(Z)&=(2\esup{g_n}{Z}+\var_Z(g_n))
             \frac{\esup{D(T^n)}{Z}}{\makebox{$m(T^nZ)$}}\,, \label{eq:AnZ} \\[2mm]
    A''_n(Z)&=   \esup{g_n}{Z}\esup{D(T^n)}{Z}\,.
     \label{eq:AAnZ}
\end{align}
\

Then for any $f\in BV(I)$ we have 
    \begin{align}
     \var_I(\tr^n_g(f\chi_Z)) & \leq 
     a_n(Z) \var_Z(f)+ A'_n(Z) \int_Z\abs{f}\,dm 
     \label{bound:var} \\
      \norm{\tr_g^n(f\chi_Z)}{L^1(I)}
      & \leq 
      A''_n(Z)\int_Z\abs{f}\,dm\,.
      \label{bound:L1}
     \end{align}
\end{lemma}
\ 

\begin{proof} 
Using (\ref{varlem:b}), (\ref{varlem:c}), and (\ref{varlem:e}) of
Lemma~\ref{lem:varlem} we have
\begin{align*}
            \var_I(\tr^n_g(f\chi_Z))
       =  &\var_I( (f\cdot g_n)\circ T_{|Z}^{-n}\cdot \chi_{T^nZ}) \\
    \leq &\var_{T^nZ}( (f\cdot g_n)\circ T_{|Z}^{-n}) + 2\esup{(f\cdot g_n)\circ T_{|Z}^{-n}}{T^nZ} \\
    \leq &\var_{Z}(f\cdot g_n) + 2\esup{f}{Z}\esup{g_n}{Z} \\
    \leq &\var_{Z}(f)\esup{g_n}{Z} +\esup{f}{Z}\var_Z(g_n)+ 2\esup{f}{Z}\esup{g_n}{Z} \\
        = &\esup{g_n}{Z}\var_Z{f} +(2\esup{g_n}{Z}+\var_Z(g_n))\esup{f}{Z}\,.
\end{align*}
But, by (\ref{varlem:d}) of Lemma~\ref{lem:varlem} we have 
  \[ \esup{f}{Z}\leq \var_Z(f)+\frac{1}{m(Z)}\int_Z\abs{f}\,dm\,,\]
and the first assertion follows by observing that 
\[ m(T^nZ)=\int_{Z}\abs{D(T^n)}\,dm\leq \esup{D(T^n)}{Z}m(Z)\,.\]
For the second we
use a change of variables formula to obtain 
\begin{align*}
            \norm{\tr_g^n(f\chi_Z)}{L^1(I)}
       =   &\int_I \abs{(f\cdot g_n)\circ T_{|Z}^{-n}\cdot \chi_{T^nZ}}\,dm\\
       =   &\int_{T^nZ} \abs{f\cdot g_n}\circ T_{|Z}^{-n}\,dm \\
       =   &\int_{Z} \abs{f}\cdot \abs{g_n}\cdot \abs{D(T^n)}\,dm \\
      \leq &\esup{g_n}{Z}\esup{D(T^n)}{Z}\int_Z\abs{f}\,dm\,.
\end{align*}
\end{proof}

We are now able to deduce the following $BV$-$L^1$ Lasota-Yorke inequality. 

\begin{proposition}
\label{prop:LYI}
Let $T:\bigcup \Zp \to I$ be a piecewise monotonic $C^1$-map and 
let $g\in BV(\bigcup\Zp )$. 
 Then we have for any $n\in\N$
\begin{equation}
   \label{eq:L1B}
    \norm{\tr^n_gf}{L^1(I)}\leq A_n\norm{f}{L^1(I)}\quad (\forall f\in BV(I))\,,
\end{equation}
\begin{equation}
   \label{eq:LYI} 
    \norm{\tr^n_gf}{BV(I)}\leq a_n \norm{f}{BV(I)} +A_n\norm{f}{L^1(I)}\quad (\forall f\in BV(I))\,,
\end{equation}
where, using the notation from the previous lemma, 
\begin{align*}
  a_n&=\sup_{Z\in \Zp_n} a_n(Z)\,,\\
  A_n&=\sup_{Z\in \Zp_n} (A'_n(Z)+A''_n(Z))\,.
\end{align*}
\end{proposition}

\begin{proof}
For any $f\in BV(I)$ and $Z\in \Zp_n$, we have by Lemma~\ref{lem:intLg}  
   \[ \norm{\tr^n_g(f\chi_Z)}{L^1(I)}\leq A_n \int_Z\abs{f}\,dm\,,\]
so
   \[ \norm{\tr^n_gf}{L^1(I)}
        \leq \sum_{Z\in \Zp_n} \norm{\tr^n_g(f\chi_Z)}{L^1(I)}
        \leq A_n \sum_{Z\in \Zp_n}\int_Z\abs{f}\,dm\\
          \leq A_n\int_I\abs{f}\,dm\,,
  \]
which proves (\ref{eq:L1B}).  
Similarly, 
  \[ \norm{\tr^n_g(f\chi_Z)}{BV(I)}\leq a_n \var_Z(f)+ A_n \int_Z\abs{f}\,dm\,,\]
so, using (\ref{varlem:f}) of Lemma~\ref{lem:varlem},   
\begin{align*}  
      \norm{\tr^n_gf}{BV(I)}
        &\leq \sum_{Z\in \Zp_n} \norm{\tr^n_g(f\chi_Z)}{BV(I)}\\
        &\leq a_n \sum_{Z\in \Zp_n}\var_Z(f)+ A_n \sum_{Z\in \Zp_n}\int_Z\abs{f}\,dm\\
        &\leq a_n \var_I(f)+ A_n \int_I\abs{f}\,dm\,,
\end{align*}  
which proves (\ref{eq:LYI}).  
\end{proof}

In the following we need a control 
on the growth of the above coefficients $a_n$
 and $A_n$ which is uniform in $g$. In order to achieve this we will use the following notation. 
Define for $n\in\N$ (recalling that $\Zp_n \neq \emptyset$ by our standing assumption):
\begin{align*}
   \Theta_n(T,g) & =  \frac1n  \sup_{Z\in \Zp_n} \log \esup{g_n}{Z}\,  \\
   \Lambda_n(T)  & = \frac1n  \sup_{Z\in \Zp_n} \log \esup{D(T^n)}{Z}\, \\
   \Xi_n(T)  & =  \frac1n \sup_{Z\in \Zp_n} \log \frac{\esup{D(T^n)}{Z}}{\makebox{$m(T^n Z)$}}
                 \,  \\
\end{align*}
and let
\begin{align*}
   \Theta(T,g)  &:= \lim_{n\to\infty} \Theta_n(T,g)  \\
   \Lambda(T)  &:= \lim_{n\to\infty}  \Lambda_n(T) \\
   \Xi (T)  &:= \limsup_{n\to\infty}  \Xi_n(T)\,,   
\end{align*}
where, as before, $D(T^n)$ denotes the derivative of $T^n$.

\begin{remark} \
\begin{enumerate}
\item The first two limits can be shown to exist 
by a simple sub-additivity argument. 
\item A short calculation shows that $\Xi(T)\geq 0$. 
\item The quantity $\Lambda(T)$
measures the growth rate of the maximal expansion of $T^n$. Although it is easier to calculate than $\Xi(T)$ it 
will not play any further role in our results. We note, however, that it is closely related to $\Xi(T)$, as we shall see presently.  
Since $m(T^n Z) \leq m(I)<+\infty$ we obviously have $\Lambda(T) \leq \Xi(T)$.
There is equality when $T$ has \emph{big images}, that is, if 
   \[ 
     \liminf_{n\to \infty} \left ( \inf \all{ m(T^nZ)}{ Z\in \Zp_n}
                 \right )^{1/n}=1\,.
   \]
Note that any piecewise monotonic $C^1$-map
which is 
Markov has big images.
\end{enumerate}
\end{remark}

The following lemma establishes uniform bounds for the constants occurring in the Lasota-Yorke inequality. 
 \begin{lemma}
\label{lem:vargn}
Suppose that $T:\bigcup \Zp \to I$ is a piecewise monotonic $C^1$-map and that $g\in BV(\bigcup \Zp)$. 
Let $\beta >\exp(\Theta(T,g))$ and define the 
   quantities (tacitly assuming that $M\geq 1$)
\begin{align}
  M = M(T,g,\beta) & :=\;  \sup_{n\geq 1}\;  \sup_{Z\in \Zp_n} 
       \; \beta^{-n} \esup{g_n}{Z}  \label{M-bound}\\
   \Gamma = \Gamma(T,g) & :=\;  \sup_{Z\in\Zp} \; \myvar{Z} (g) \, .
\end{align}
Then for $n\in \N$ and $Z\in \Zp_n$
\begin{align*}
        \esup{g_n}{Z} & \leq M \beta^n \,, \\ %
  \ds \myvar{Z} (g_n)
      & \leq n \; \Gamma \; M^2 \;  \beta^{n-1}\,. 
         \end{align*}
 \end{lemma}

\begin{proof} The first equation is
obviously a consequence of (\ref{M-bound}).
Let $n\in \N$ and $Z\in \Zp_n$ with 
  \[ Z=Z_0\cap T^{-1} Z_1\cap \ldots \cap T^{-(n-1)}Z_{n-1}\,, \]
for some $Z_i\in \Zp$. Then, using (\ref{varlem:e}) of 
Lemma \ref{lem:varlem}, we have  
\begin{align*}
      \var_Z(g_n)&\leq \sum_{k=1}^{n} \esup{g_{k-1}}{Z}
              \var_Z(g\circ T^{k-1})\esup{g_{n-k}\circ T^k}{Z}\\
                &\leq \sum_{k=1}^n
                M  \beta^{k-1} \var_{T^{k-1}Z}(g)\, M \beta^{n-k}\\
                &\leq 
                 \sum_{k=1}^n\var_{T^{k-1}Z}(g)
                \; M^2\; \beta^{n-1}
                \,.
\end{align*}
But $T^{k-1}Z\subset Z_k$, so 
  \[ \sum_{k=1}^n\var_{T^{k-1}Z}(g)
    \leq \sum_{k=1}^n\var_{Z_k}(g)
    \leq n \; \Gamma \]
and the assertion follows.
\end{proof}

The following estimate will provide
a perturbative bound when varying the weight.

\begin{lemma}
\label{lem:lg_kl3}
Let $T:\bigcup \Zp \to I$ be a piecewise monotonic $C^1$-map. Then for
any $g\in L^1(\bigcup \Zp)$ and any $f\in BV(I)$ we have   
   \[ \norm{\tr_g f}{L^1(I)} 
\leq \frac{\exp(\Lambda_1(T))}{\min\set{1,m(I)}} \norm{g}{L^1(\bigcup \Zp)} \norm{f}{BV(I)}\,.\] 
\end{lemma}

\begin{proof}
Using the change of variables formula we have, for $Z\in \Zp$,  
\begin{align*} 
            \norm{\tr_g(f\chi_Z)}{L^1(I)}
       =   &\int_I \abs{(f\cdot g)\circ T_{|Z}^{-1}\cdot \chi_{TZ}}\,dm\\
       =   &\int_{TZ} \abs{f\cdot g}\circ T_{|Z}^{-1}\,dm \\
       =   &\int_{Z} \abs{f}\cdot \abs{g}\cdot \abs{D(T)}\,dm \\
      \leq &\esup{f}{Z}\esup{D(T)}{Z}\int_Z\abs{g}\,dm\,.
\end{align*}
But $\esup{D(T)}{Z}\leq \exp(\Lambda_1(T))$ and by (\ref{varlem:d}) of 
Lemma~\ref{lem:varlem} 
  \[ \esup{f}{Z}\leq \esup{f}{I} 
                      \leq \var_I(f) + \frac{1}{m(I)} \int_I \abs{f}\,dm
                      \leq  \frac{1}{\min\set{1,m(I)}} \norm{f}{BV(I)}\,, \]
so the inequality follows since 
  \[   \norm{\tr_gf}{L^1(I)}  
           \leq \sum_{Z\in \Zp}\norm{\tr_g(f\chi_Z)}{L^1(I)} 
           \leq  \frac{\exp(\Lambda_1(T))}{\min\set{1,m(I)}}\norm{f}{BV(I)} \sum_{Z\in \Zp}\int _Z \abs{g}\,dm\,.\]
\end{proof}

\section{The Keller-Liverani perturbation theorem and its consequences}
\label{sec:KL}
Let $(X,\norm{\cdot}{})$ be a Banach space and $L:X\to X$ a bounded linear operator. In the following, we 
denote by $\varrho(L)$ the \emph{resolvent set} of $L$
    \[ \varrho(L)=\all{z\in \C}{\text{$zI-L$ is boundedly invertible on $X$}}\,\]
and by $\sigma(L)=\C\setminus \varrho(L)$ the \emph{spectrum} of $L$. 
Recall that the \emph{discrete spectrum} $\sigma_{\rm disc}(L)$ consists of all isolated eigenvalues of 
$L$ with finite algebraic multiplicity, and that its complement in $\sigma(L)$ is known as the 
\emph{(Browder) essential spectrum} $\sigma_{\rm ess}(L)=\sigma(L)\setminus \sigma_{\rm disc}(L)$. 
The \emph{spectral radius} of $L$ will be denoted by $\rho(L)$ and the \emph{essential spectral radius} of $L$ 
by $\rho_{\rm ess}(L)$, that is, 
 \begin{align*}
  \rho(L)&=\sup\all{\abs{\lam}}{\lam\in \sigma(L)}\,, \\
  \rho_{\rm ess}(L)&=\sup\all{\abs{\lam}}{\lam\in \sigma_{\rm ess}(L)}\,.
 \end{align*}
 Finally, given   
$\delta>0$ and $r>0$ we write 
  \[ W_{r,\delta}(L)=\all{z\in \C}{\abs{z}>r \text{ and } \dist(z,\sigma(L))>\delta}\,,\]
and 
  \[ \sigma_r(L)=\all{z\in \C}{\abs{z}\leq r}\cup \sigma(L)\,.\]

The Keller-Liverani Theorem applies in the following situation. 
Suppose that the Banach space $(X,\norm{\cdot}{})$ is equipped with a second norm $\wnorm{\cdot}$, which is 
weaker 
than the original norm in the sense that $\wnorm{f}\leq \norm{f}{}$ for every $f\in X$. Given a bounded linear 
operator $L:X\to X$ we write 
  \[ \tnorm{L}=\sup\all{\wnorm{Lf}}{f\in X, \norm{f}{}\leq 1} \]
for the norm of $L$ considered as an operator from $(X,\norm{\cdot}{})$ to $(X,\wnorm{\cdot})$. 

Furthermore, suppose that $E=[0,\eps']$ is a set of parameter values where $\eps'$ is a positive constant and that we are 
given a family 
$(L_\eps)_{\eps\in E}$ of bounded linear operators on $(X,\norm{\cdot}{})$ satisfying the following three properties. 
\begin{itemize}
\item[(KL1)]
There are constants $C_1, C_2, C_3>0$ and $0<a<A$ such that for every $\eps\in E$
\begin{equation}
   \label{KL:UB}
    \wnorm{L^n_\eps f} \leq C_1A^n\wnorm{f}   (\forall f\in X,  \forall n\in \N)\,,
\end{equation}
and 
\begin{equation}
   \label{KL:LY}
   \norm{L^n_\eps f}{} \leq C_2a^n \norm{f}{} + C_3A^n\wnorm{f}   (\forall f\in X, \forall n\in \N)\,. 
\end{equation}
\item[(KL2)] For every $\eps\in E$, if $\lam\in \sigma(L_\eps)$ and $\abs{\lam}> a$  
    then $\lam$ does not belong to the residual spectrum of $L_\eps$. 
\item[(KL3)] There is a monotonic upper-semicontinuous function $\tau:E\to [0,\infty)$ with 
    $\lim_{\eps\downarrow 0}\tau_\eps=0$
    and 
       \[ \tnorm{L_\eps-L_0}\leq \tau_\eps \quad (\forall \eps \in E)\,.\]
\end{itemize}

\begin{remark}
\label{rem:ITM}
Note that condition (KL2) follows from condition (KL1) provided that the embedding $(X,\norm{\cdot}{})\hookrightarrow (X,\wnorm{\cdot})$ is compact, since then 
\[ \rho_{\rm ess}(L_\eps)\leq a \quad (\forall \eps\in E)\,,\]
by a theorem of Ionescu Tulcea and Marinescu (see \cite{ITM} for the original paper or \cite{HH} for a more contemporary exposition). 
\end{remark}

The Keller-Liverani Theorem can now be stated as follows. 
\begin{theorem}
\label{thm:KL}
Suppose that $(L_\eps)_{\eps \in E}$ is a family of bounded linear operators satisfying conditions (KL1), (KL2), 
and (KL3) above. Fix $r\in (a,A)$ and $\delta>0$. Then there is a constant $\eps_0=\eps_0(r,\delta)>0$ 
such that 
\begin{equation}
  \label{KL:eq:spec} 
   W_{r,\delta}(L_0)\subset \varrho(L_\eps) \quad(\forall \eps\in [0,\eps_0])\,.
\end{equation}
Moreover, if ${\mathcal C}$ is a compact subset of $W_{r,\delta}(L_0)$, then there is a constant 
  $K_0=K_0(r,\delta,{\mathcal C})>0$ such that 
\begin{equation}
  \label{KL:eq:res}
  \sup_{z\in {\mathcal C}}\tnorm{(zI-L_\eps)^{-1}-(zI-L_0)^{-1}}\leq K_0 \tau_\eps^\eta 
      \quad (\forall \eps \in [0,\eps_0])\,,
\end{equation}
where 
  \[ \eta=\frac{\log(r/a)}{\log(A/a)}\,.\]
\end{theorem}

\begin{proof} See \cite[Theorem~1]{KL99}. 
\end{proof}

Suppose now that $\lam\in \C$ with $\abs{\lam}>a$ is an isolated point of the spectrum of $L_0$. Fixing $r$ with 
$a<r<\abs{\lam}$ we can choose $\delta>0$ so that $\Delta_{\delta}(\lam)\cap \sigma_r(L_0)=\set{\lam}$, 
where 
\[ \Delta_{\delta}(\lam)=\all{z\in \C}{\abs{z-\lam}\leq \delta} \]
denotes the closed disk with radius $\delta$ and centre $\lam$. 
The Keller-Liverani Theorem now implies that for all $\eps$ sufficiently small, the positively-oriented 
boundary $\partial \Delta_{\delta}(\lam)$ belongs to $\varrho(L_\eps)$ implying that 
the spectral  projections 
  \[ P_\eps^{(\lam,\delta)}=\frac{1}{2\pi \rm{i}} \oint_{\partial \Delta_{\delta}(\lam)} (zI-L_\eps)^{-1} \,dz \] 
exist for all $\eps$ sufficiently small (e.g. for all $\eps\in [0,\eps_0(\delta/2, r)]$). In fact, more is true.  

\begin{corollary} 
\label{coro:KLcoro}
Suppose that $(L_\eps)_{\eps \in E}$ is a family of bounded linear operators satisfying conditions (KL1), (KL2), 
and (KL3) above and suppose that $\lam$ is an isolated spectral point of $L_0$ with $\abs{\lam}>a$. Fix $r$ with 
$a<r<\abs{\lam}$. Then there is  a constant $\delta_1=\delta_1(r)>0$ such that 
for every 
$\delta\in (0,\delta_1]$ there are constants $K_1=K_1(r,\delta)>0$ and $\eps_1=\eps_1(r,\delta)>0$ with the following 
properties: 

\begin{equation}
\label{KL:coro1}
   \partial \Delta _\delta (\lam)\subset \varrho(L_\eps) \quad (\forall \eps\in [0,\eps_1])\,,
\end{equation}

\begin{equation}
\label{KL:coro2}
    \norm{P_\eps^{(\lam,\delta)}f}{}\leq K_1\wnorm{P_\eps^{(\lam,\delta)}f} 
        \quad (\forall f\in X, \forall \eps \in [0,\eps_1])\,,
\end{equation}

\begin{equation}
\label{KL:coro3}
     {\rm rank}\,P_\eps^{(\lam,\delta)}= {\rm rank}\,P_0^{(\lam,\delta)} \quad (\forall \eps\in [0,\eps_1])\,.
\end{equation}
\end{corollary}

\begin{proof} See \cite[Corollary~1]{KL99}.
\end{proof}

For the applications we have in mind we require the following refinement of the previous results, providing a 
quantitative bound for the behaviour of a peripheral eigenvalue $\lambda$ of $L_0$ under perturbations which are small in the norm $\tnorm{\cdot}$. 
It turns out that
the behaviour is governed by the size 
of the largest Jordan block of $L_0$ corresponding to $\lambda$, or, equivalently, 
by the order of the pole of the resolvent of $L_0$ at $\lambda$, just as in standard perturbation theory
(see, for example, \cite[Theorem~6.7]{chatelin}). 
\begin{corollary}
\label{coro:ourcoro}
Suppose that $(L_\eps)_{\eps \in E}$ is a family of bounded linear operators satisfying conditions (KL1), (KL2), 
and (KL3) above and suppose that $\lam$ with $\abs{\lam}>a$ is a pole of the resolvent of $L_0$ of order $p$. 
Fix $r$ with $a<r<\abs{\lam}$. Then there is  a constant $\delta_2=\delta_2(r)>0$ such that 
for every $\delta\in (0,\delta_2]$ there are constants  $K_2=K_2(r,\delta,p)>0$ and $\eps_2=\eps_2(r,\delta)>0$ with the following properties:

\begin{equation}
  \label{KL:ourcoro2}
  \Delta_\delta(\lam)\cap \sigma_r(L_\eps)\not =\emptyset \quad (\forall \eps\in [0,\eps_2]) \,,
\end{equation}

\begin{equation}
   \label{KL:ourcoro3}  
    \sup\all{ \abs{\lam'-\lam}  }{\lam'\in  \Delta_\delta(\lam)\cap \sigma_r(L_\eps)} 
         \leq K_2 \tau_\eps^{\eta/p} \quad (\forall \eps \in [0,\eps_2])\,,
\end{equation}
where 
\[ \eta=\frac{\log(r/a)}{\log(A/a)}\,.\]         
 \end{corollary} 

\begin{proof} Fix $r$ with $a<r<\abs{\lam}$. Choose $\delta_2(r)\leq \delta_1(r)$ so that 
 $\Delta_{\delta_2}(\lam)\cap \sigma_r(L_0)=\set{\lam}$. 
Now fix $\delta \in (0,\delta_2)$. 
Write $\eps_2(r,\delta)=\min\set{\eps_0(r,\delta),\eps_1(r,\delta)}$ and let $\eps\in [0,\eps_2]$. 

The first assertion now follows, since we have  
$  {\rm rank}\,P_\eps^{(\lam,\delta)}= {\rm rank}\,P_0^{(\lam,\delta)}>0$ by (\ref{KL:coro3}) of the previous corollary. 

In order to prove the remaining assertion write 
  \[ \Pi=\frac{1}{2\pi \rm{i}} \oint_{\partial \Delta_{\delta}(\lam)} 
      \left [ (z-\lam)^p(zI-L_\eps)^{-1} - (z-\lam)^p(zI-L_0)^{-1} \right ] \,dz\,, \] 
and observe that $\Pi$ is a bounded linear operator on $X$.  In order to conclude the proof we 
shall bound $\tnorm{\Pi}$ from above and below. We start with the upper bound. By 
Theorem~\ref{thm:KL} we have 
with  $K_0=K_0(r,\delta, \partial \Delta_\delta(\lam))$
\begin{equation}
\label{KL:ourcoro:upper}
  \tnorm{\Pi}\leq \frac{1}{2\pi} \oint_{\partial \Delta_{\delta}(\lam)} 
                         \delta^p\tnorm{(zI-L_\eps)^{-1} - (zI-L_0)^{-1}}\,\abs{dz} \\
                   \leq \delta^{p+1} K_0 \tau_\eps^\eta\,.
\end{equation}
For the lower bound, we first use the fact that 
$\lam$ is a pole of order $p$ of $z\mapsto (zI-L_0)^{-1}$ together with analytic 
functional calculus for $L_\eps$ (see, for example, \cite[Theorem~VII.3.10]{DS}) to conclude that 
\[ \Pi = \frac{1}{2\pi \rm{i}} \oint_{\partial \Delta_{\delta}(\lam)} 
            (z-\lam)^p(zI-L_\eps)^{-1} \,dz
         =(L_\eps-\lam I)^pP_\eps^{(\lam,\delta)}\,.
\] 
Let now $\lam'\in   \Delta_\delta(\lam)\cap \sigma_r(L_\eps)$. Then there is a non-zero $f'\in X$ with 
 $P_\eps^{(\lam,\delta)}f'=f'$ and $(L_\eps-\lam')f'=0$.  Thus 
\[ \Pi f'=(L_\eps-\lam I)^pf'
          =\left [ (L_\eps-\lam' I)+(\lam'-\lam)I \right ]^pf'
          =(\lam'-\lam)^pf'
          =(\lam'-\lam)^pP_\eps^{(\lam,\delta)}f'\,.
\]          
Since $f'\neq 0$ we may, with no loss of generality, assume that $\norm{f'}{}=1$. Then, using the previous equation together with (\ref{KL:coro2}) of the previous corollary, we have 
\begin{multline}
\label{KL:ourcoro:lower} 
   \tnorm{\Pi}\geq \wnorm{\Pi f'}
          =    \abs{\lam'-\lam}^p\wnorm{P_\eps^{(\lam,\delta)}f'}
        \geq \abs{\lam'-\lam}^pK_1^{-1}\norm{P_\eps^{(\lam,\delta)}f'}{}\\
          =    \abs{\lam'-\lam}^pK_1^{-1}\norm{f'}{}
          =    \abs{\lam'-\lam}^pK_1^{-1}\,.
\end{multline}
Combining (\ref{KL:ourcoro:upper}) and (\ref{KL:ourcoro:lower}) we have 
\[ \abs{\lam'-\lam}^p\leq \delta^{p+1} K_0K_1 \tau_\eps^\eta \]
and (\ref{KL:ourcoro3}) follows by setting 
$K_2=\delta^{1+1/p}(K_0 K_1)^{1/p}$.
\end{proof}

\section{Local H\"older continuity of the pressure}
\label{sec:holder}
In this section we prove our main abstract result on the regularity of the pressure. Before we start with a 
definition of pressure suitable for our purposes, we require the following notation. 
For $J$ a finite union of open subsets of $J$ we write $BV_+(J)$ for the space of real
non-negative weights on $J$ of bounded variation. 

\begin{definition}
Suppose that $T:\bigcup \Zp\to I$ is a piecewise monotonic $C^1$-map. 
Let $g\in BV_+(\bigcup \Zp)$  and 
let $\tr_g:BV(I)\to BV(I)$ denote the corresponding transfer operator. 
We then call 
   \[ P(T, g)=\log \rho (\tr_g) \]
the \emph{pressure of $T$ with weight $g$}.  
\end{definition}

An alternative expression for the pressure is given in the following theorem. 

\begin{theorem} \label{thm:bkpressure}
Let $T:\bigcup \Zp\to I$ be a piecewise monotonic $C^1$-map and $g\in
BV_+(\bigcup \Zp)$.  If either
\begin{enumerate}
\item
$g_{|Z}$ is continuous 
for each $Z\in \Zp$ and $\Zp$ is generating, or
\item
$g_{|Z}$ is constant
for each $Z\in \Zp$ ($\Zp$ need not be generating)
\end{enumerate}
then 
   \[ P(T,g)=\limsup_{n\to \infty} \frac{1}{n}
      \log \sum_{Z\in \Zp_n} \sup_Z g_n\,.\]
\end{theorem}
\begin{proof}
The first case is proved as in \cite[Theorem~3]{BK}. In the second case
we note that $g_{n|Z}$ is constant for each $Z\in\Zp_n$.
The set $\calA_n = \{ T^n Z : Z\in \Zp_n\}$ is a subset of the so-called
Hofbauer tower so has cardinality 
$\Card \calA_n \leq 2 n \Card \Zp_1$ (see, for example, \cite[Section~3]{BK}).   
For each
$A\in\calA_n$, pick $x_A\in A$. Then
\[ \Omega_n
    := \sum_{Z\in \Zp_n} g_{n|Z}
    \leq  \sum_{A\in\calA_n} \sum_{y\in T^{-n}x_A} g_n(y) 
    = \sum_{A\in\calA_n} \left( \tr_g^n  \bfone \right) (x_A),\]
so $\Omega_n \leq 2n \Card Z_1 \|\tr_g^n \|$. It
follows that $\limsup_{n\to \infty} \Omega_n^{1/n} \leq \rho(\tr_g)$.
The reverse inequality follows as in the proof of
\cite[Theorem~3]{BK}.
\end{proof}

Before we continue we summarise some consequences of the results in Section~\ref{sec:setup}. 
We start with the following well-known facts already contained in \cite{Keller, BK}, the simple proofs 
of which we give for the convenience of the reader. 

\begin{lemma}
\label{lem:ress}
Let $T:\bigcup \Zp \to I$ be a piecewise monotonic $C^1$-map and 
$g\in BV(\bigcup \Zp)$. Then 
\[ \rho_{\rm ess}(\tr_g)\leq \exp(\Theta(T,g))\,.\]
\end{lemma}
\begin{proof}
Fix $\beta>\exp(\Theta(T,g))$. Proposition~\ref{prop:LYI} and Lemma~\ref{lem:vargn} imply that 
the transfer operator $\tr_g$ satisfies 
a Lasota-Yorke inequality of the form 
\[ \norm{\tr_g^nf}{BV(I)} \leq a_n \norm{f}{BV(I)} + A_n\norm{f}{L^1(I)} \quad (\forall n\in \N, \forall f \in BV(I))\,,\]
 with
 \[ \limsup_{n\to\infty} a_n^{1/n} \leq \beta\,.\]

Since the embedding $BV(I)\hookrightarrow L^1(I)$ is compact, 
it follows, for example by \cite{HH}, that $\rho_{\rm ess}(\tr_g)\leq \beta$.  
Since $\beta>\exp(\Theta(T,g))$ was arbitrary, the assertion is proved. 
\end{proof}

\begin{lemma}
Let $T:\bigcup \Zp \to I$ be a piecewise monotonic $C^1$-map and 
$g\in BV_+(\bigcup \Zp)$. 
If $P(T, g)> \Theta(T,g)$,  
then $\exp(P(T,g))$ is a pole of the resolvent of $\tr_g$ and in
particular an eigenvalue of $\tr_g$.  
\end{lemma} 

\begin{proof}
This follows from a positivity argument. 
Consider the real Banach space $BV(I)_\R=\all{f:I\to \R}{
  \norm{f}{BV(I)}<\infty }$ and 
observe that 
$BV_+(I)$ is a total cone, that is, a cone the span of which is dense in $BV(I)_\R$. 
Since $g\geq 0$ the transfer operator leaves this cone invariant. 
Moreover, since  $P(T, g)> \Theta(T,g)$, the previous lemma implies 
that $\rho(\tr_g)>\rho_{\rm ess}(\tr_g)$, so the 
resolvent of $\tr_g$ must have a pole on the circle 
centred at $0$ with radius 
$\rho(\tr_g)$. Using \cite[Appendix 2.4]{schaefer}) 
it now follows that $\rho(\tr_g)$ is a pole of the resolvent of 
$\tr_g$ on $BV(I)_\R$ and the assertion of the lemma is 
proved since the complex Banach space $BV(I)$
is equivalent to any
standard complexification of $BV(I)_\R$.
 \end{proof}

\begin{remark}
This also follows from \cite[Theorem~2]{BK}. The above proof, however, is more direct. 
\end{remark}

\begin{definition}
\label{def:Holder}
When $(X,d_X)$ and $(Y,d_Y)$ are metric spaces, a function
$h: X \rr Y$
is  said to be 
\emph{H\"older continuous at $t\in X$}
if there is $\alpha\in (0,1]$ such that 
  \[ \limsup_{s\to t} \frac{ d_Y(h(s),h(t)) }{d_X(s,t)^\alpha}<\infty\,.\]   
In case such an $\alpha$ exists, we call 
  \[ \Hol(h,t)= 
        \sup\all{\alpha\in (0,1]}{%
                 \limsup_{s\to t}
                 \frac{ d_Y(h(s),h(t))}{d_X(s,t)^\alpha}<\infty}
  \]
the \emph{local H\"older exponent of $h$ at $t$}. 
\end{definition}

In the following we consider a fixed piecewise
monotonic $C^1$-map $T:\bigcup \Zp\to I$ and a family 
$(g_s)_{s\in \mathcal{S}}$ of weights.  
We shall be interested in the regularity of $s\mapsto P(T,g_s)$. 
The following is our abstract main result. 

\begin{theorem}
\label{thm:main}
Suppose that $T:\bigcup \Zp\to I$ is a 
piecewise monotonic  $C^1$-map  with $0 < \Xi(T)<+\infty$.
Let $t\in \R$ and let $\mathcal S$ be a neighbourhood of $t$. 
Suppose that $(g_s)_{s\in \mathcal{S}}$ is a family of 
non-negative weights on $\bigcup \Zp$ of bounded variation
and  that there is $\Theta \in \R$
with the following properties
    (see Lemma \ref{lem:vargn} for definitions):
\begin{itemize}
   \item[(i)] 
       $\ds \limsup_{s\to t} \Theta(T,g_s) \leq \Theta$ 
          and 
          $\ds \limsup_{s\to t} \ \Gamma(T,g_s)<+\infty$;\\
   \item[(ii)] 
          $\ds \limsup_{s\to t} \
            M(T,g_s,\beta)<+\infty $
   whenever $\beta>\exp \Theta$;\\
   \item[(iii)] $\norm{g_s-g_t}{L^1(\bigcup \Zp)}\leq \tau_{s-t}$, where $\tau$
     is a monotonic upper-semicontinuous function defined in a
     neighbourhood of $0$ with $\lim_{\eps\to 0}\tau_\eps=0$. 
\end{itemize}
      If $\Theta < P(T,g_t)$,
   then for every 
      \begin{equation}  
          0< \eta< 
    \eta_0 := \frac{P(T,g_t)-\Theta}{\Xi(T)} 
    \end{equation}
    there is a constant $K'$ and a
neighbourhood $\mathcal{S'}$ of $t$ such that
   \begin{equation}
   \abs{P(T,g_s)-P(T,g_t)}\leq
   K' \tau_{s-t}^{\eta/p} \quad (\forall s\in \mathcal{S'})\,,
   \label{holderineq}
   \end{equation}
where 
 $p$ is the order of the pole of the resolvent of $\tr_{g_{t}}$ at
 $\exp(P(T,g_t))$. \\
\end{theorem}

\begin{proof}
Fix $\eta<\eta_0$. We may choose  constants satisfying
$\Theta  < \Theta' < R < P(T,g_t)$ and $\Xi'>\Xi(T)$ such that 
$\eta=\frac{R-\Theta'}{\Xi'}$. Now fix $\beta$ with 
$\exp(\Theta) < \beta < \beta':=\exp(\Theta')$. 

By possibly shrinking $\mathcal S$ we may, by (i) and (ii), assume that
$\sup_{s\in \mathcal S} M(T,g_s,\beta) \leq M < +\infty$ and
$\sup_{s\in \mathcal S} \Gamma(T,g_s) \leq \Gamma < +\infty$.
We thus have 
   \[ \esup{g_{s,n}}{Z} \leq M \exp( n\Theta' ) \quad (\forall s\in \mathcal{S}, \forall n\in \N, \forall Z\in \Zp_n)\,. \]
Moreover, 
by Lemma \ref{lem:vargn} we have 
$\var_{Z} (g_{s,n}) \leq n \Gamma M^2 \beta^{n-1}$ for all 
$s\in \mathcal S$ and all $n\in\N$, 
$Z\in \Zp_n$. 
As $\beta'>\beta$ we may absorb the factor $n$ and conclude that there is 
a constant $C'$ not depending on $s$ such that 
    \[ \var_Z (g_{s,n}) \leq C' \exp(n\Theta')  \quad (\forall s\in \mathcal{S}, \forall n\in \N, \forall Z\in \Zp_n)\,. \]
Furthermore, by the definition of $\Xi(T)$ there is a constant $C''$ such that 
    \[ \frac{\esup{D(T^n)}{Z}}{\makebox{$m(T^nZ)$}}\leq C'' \exp(n\Xi')  
           \quad (\forall n\in \N, \forall Z\in \Zp_n)\,.\]
 Hence we also have 
    \[ \esup{D(T^n)}{Z} \leq m(I)\,C'' \exp(n\Xi')  
           \quad (\forall n\in \N, \forall Z\in \Zp_n)\,.\]
We can thus bound the constants $a_n$ and $A_n$ occurring in 
Proposition~\ref{prop:LYI} by  
\[ a_n\leq C_2 \exp(n\Theta')\,, \quad A_n\leq C_3 \exp(n\Theta')\exp(n\Xi') \]
where $C_2=3M+C'$ and $C_3=(2M+m(I)M+C')C''$, which shows that  
condition (KL1) of the Keller-Liverani Theorem is 
satisfied with 
$a=\exp(\Theta')$  and $A=\exp(\Theta'+\Xi')$ for the family 
$(\tr_{g_{t+\eps}})_{\eps\in E}$ with $E$ a small neighbourhood of $0$. 

Condition (KL2) follows from Remark~\ref{rem:ITM} together with the 
compactness of the embedding $BV(\bigcup \Zp) \hookrightarrow L^1(\bigcup \Zp)$. 

By Lemma~\ref{lem:lg_kl3}, there is a constant $C$ such 
that 
    \[ \tnorm{\tr_{g_s}-\tr_{g_t}}{}=\tnorm{\tr_{g_s-g_t}}{}
    \leq C\norm{g_s-g_t}{L^1(\bigcup \Zp)} \leq C \tau_{s-t}\,, \]
so condition (KL3) of the Keller-Liverani Theorem is also satisfied. 

Since $\lam_s=\exp(P(T,g_s))$ is the largest real  
eigenvalue of $\tr_{g_s}$ we have 
$\lim_{s\to t}\lam_s=\lam_t$. 
  In order to see this, fix $r=\exp(R)$ and let $\delta\in (0, (\lam_t-r)/2)$. Observing 
  that by 
  (\ref{KL:eq:spec}) of 
  Theorem~\ref{thm:KL}
  we have $\lambda_s\leq \lambda_t+\delta$ whenever
  $|s-t|\leq \eps_0(r,\delta)$ as $\sigma(\tr_{g_s})$ is contained   
  in a $\delta$-neighbourhood of $\sigma_r(\tr_{g_t})$.
  Similarly,
  by (\ref{KL:coro3}) of Corollary~\ref{coro:KLcoro} we have
  $\lambda_t-\delta\leq \lambda_s$ for 
  $|s-t|\leq \eps_1(r,\delta)$ since there exists an eigenvalue
  of $\tr_{g_s}$ in a $\delta$-neighbourhood of $\lambda_t$.

To conclude the proof we again set $r = \exp(R)$ 
and observe that, by Corollary~\ref{coro:ourcoro}, 
we can choose a $\delta>0$ and a neighbourhood 
$\mathcal{S'}$ of $t$ so that 
$\lam_s\in \Delta_\delta(\lam_{t})$ for all $s\in \mathcal{S'}$ and 
   \[ \abs{\log \lam_s-\log \lam_{t}} 
    \leq \frac{1}{\min\set{\lam_s,\lam_t}} \abs{\lam_s-\lam_t} 
    \leq \exp(-\Theta ) K_2 (C \tau_{t-s})^{\eta/p} 
     \quad (\forall s\in \mathcal{S'})\,, 
   \] 
where 
   \[ \eta=\frac{R-\Theta'}{\Xi'}\,,\]
and $K_2$ is the constant in Corollary~\ref{coro:ourcoro}. 
\end{proof}

\begin{corollary}
\label{coro:main}
Under the hypotheses of the previous theorem:
 \begin{enumerate}[(a)]
   \item \label{coro:main:a}
     if $s\mapsto g_s\in L^1(\bigcup \Zp)$ is 
         continuous at $t$, 
          then so is the map $s\mapsto \max\{P(T,g_s), \Theta\}$;
   \item \label{coro:main:b}
     if  $P(T,g_t)>\Theta$ and $s\mapsto g_s\in L^1(\bigcup \Zp)$ is 
         Lipschitz continuous at $t$
         then $s\mapsto P(T, g_{s})$ is 
         H\"older continuous at $t$ with local H\"older exponent 
         satisfying 
          \[ \Hol(P(T,g), t)\geq 
                  \frac{P(T,g_{t})-\Theta }{ p  \; \Xi(T)}\,, \]  
         where $p$ is the order of the pole of the resolvent 
         of $\tr_{g_{t}}$ at $\exp(P(T,g_t))$. 
 \end{enumerate}
\end{corollary}
\begin{proof}
Part (b) follows immediately from (\ref{holderineq}) in our
previous theorem by taking suitable limits.
In order to show (a) it suffices to prove that if 
$P(T,g_t) \leq \Theta$ then 
also $\limsup_{s\to t} P(T,g_s) \leq \Theta$.
Hence, assume that  
$P(T,g_t) \leq \Theta$ and let $\varepsilon \in (0,\Xi(T))$. The spectrum of
$\tr_{g_t}$ is contained in the disk $\all{z\in \C}{\abs{z}\leq \exp(\Theta)}$. 
Choosing $r=\exp(\Theta+\varepsilon)$ 
and $\delta >0$ so that $\exp(\Theta)+\delta \leq \exp(\Theta + \varepsilon)$ in 
Theorem~\ref{thm:KL}, inclusion (\ref{KL:eq:spec}) implies that 
  if $|s-t|\leq \eps_0(r,\delta)$ then 
  $\sigma(\tr_{g_s})$ is contained in the disk $\all{z\in \C}{\abs{z}\leq \exp(\Theta+\varepsilon)}$. 
  So $P(T,g_s)\leq \Theta+\varepsilon$
  for such $s$-values and the claim follows.
\end{proof}

\section{Application: regularity of topological entropy for 
   interval maps with holes}
\label{sec:htop}

We shall now apply the results of the previous section to interval maps with holes and, in particular, prove 
Theorem \ref{theorem:Main} from the introduction.
 In the following, we let $T:\bigcup\Zp \to I$ be a fixed piecewise monotonic $C^1$-map.
The space of holes equipped with
a pseudometric was defined in the introduction, see 
equations (\ref{def:holespace}) and (\ref{def:holedist}).

First we note that the relation between the pressure and the (non-negative) topological entropy 
of the dynamical system with holes is as follows. 
\begin{proposition}
Let $H\in \Ho(I)$. Then 
  $\ds \htop(T,H)=\max\{ P(T, \chi_{I\setminus H}), 0\}$.
\end{proposition}  
\begin{proof}
Given the weight $g=\chi_{I\setminus H}$, it is clear that $g_n$ takes values one or zero only. 
For $Z\in \Zp_n$, we have $\sup_Z g_n=1$ precisely when $Z$ contains a non-empty cylinder for the 
restricted dynamics, say $Z'\in \Zp^H_n$. In fact,
the support of $g_n$ restricted to $Z$ is the closure of $Z'$.
Thus $\sum_{Z\in \Zp_n} \sup_Z g_n= \Card \Zp^H_n$
and by Theorem~\ref{thm:bkpressure} it follows that    
when $\Zp^H_n$ is non-empty for every $n\in \N$ 
\[ 
   P(T, \chi_{I\setminus H})= \limsup_{n\to \infty} \frac{1}{n}
      \log \sum_{Z\in \Zp_n} \sup_Z g_n =
   \limsup_{n\to \infty} \frac{1}{n} \log^+  \Card \Zp^H_n=\htop(T,H)\geq 0\,,
\]
and the assertion follows in this case.  

If on the other hand $\Zp^H_n$ is empty for some $n\in \N$, then $\Zp^H_\nu$ is empty for all $\nu\geq n$, hence  $\htop(T,H)=0$. Moreover, $g_\nu=0$ for all $\nu\geq n$, so $\rho( \tr_g )=0$ and the assertion follows in this case as well. 
\end{proof}

\begin{proof} [Proof of Theorem \ref{theorem:Main}:]

Since the weight $g_s=\chi_{I\setminus H_s}$ takes the values
0 and 1 only,
we observe that
\begin{equation}
  \label{eq:thetazero}
  \Theta(T,\chi_{I\setminus H})
     =0  \ \  (\mbox{or} \  -\infty)
     \quad (\forall H \in \Ho(I))\,,
\end{equation}
\begin{equation}
  \label{eq:disth1h2}
    \norm{g_{s}-g_{t}}{L^1(\bigcup \Zp)}
   = \dist(H_s,H_t)  \quad (\forall H_s, H_t \in \Ho(I))\,,
 \end{equation}
so the family of weights is 
$L^1$-Lipschitz continuous in $s$.
If the number of holes is uniformly bounded by $N<+\infty$ 
and the cardinality of $\Zp$ is $d$ then
the variation norm of $g_s$ is uniformly bounded by $2N+2d$.

We therefore have 
$M(T,g_s,\beta) \leq 1$ whenever $\beta>1$ and $\Gamma(T,g_s) \leq 2N+2d$, 
both uniformly in $s$. The conditions in Theorem \ref{thm:main}
are thus satisfied with $\Theta=\log(1)=0$.
The conclusions follow by applying
 Corollary~\ref{coro:main} to 
  $\htop(T,H_s)=\max\{ P(T, g_s), 0\}$.
\end{proof} 

This result is perhaps a bit surprising, given the rather brutal nature of introducing a hole 
into the dynamics.
Note  that $P(T,g_s)$ may drop discontinuously
from $0$ to $-\infty$, which happens when all points escape. 
This is the reason for defining the topological entropy as in the introduction. 

\section{Examples}
\label{sec:ex}

 \begin{figure}
  \hspace{-5em}
\includegraphics[width=0.15\paperwidth, bb=20 20 575 575]{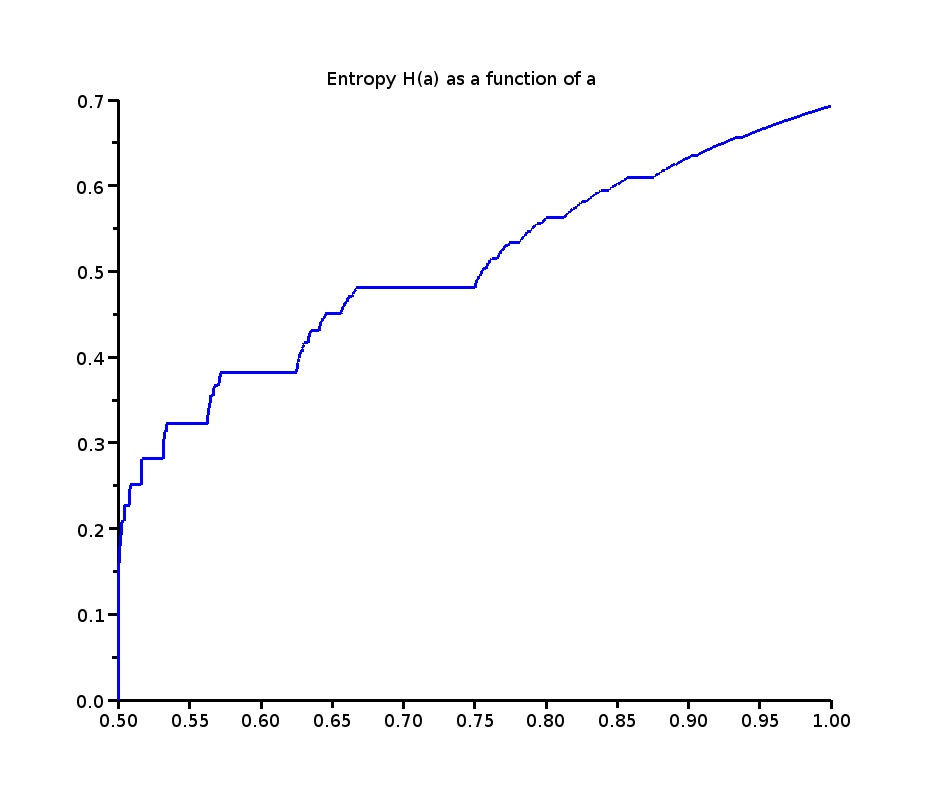} \hspace*{2cm}
\includegraphics[width=0.15\paperwidth, bb=20 20 575 575]{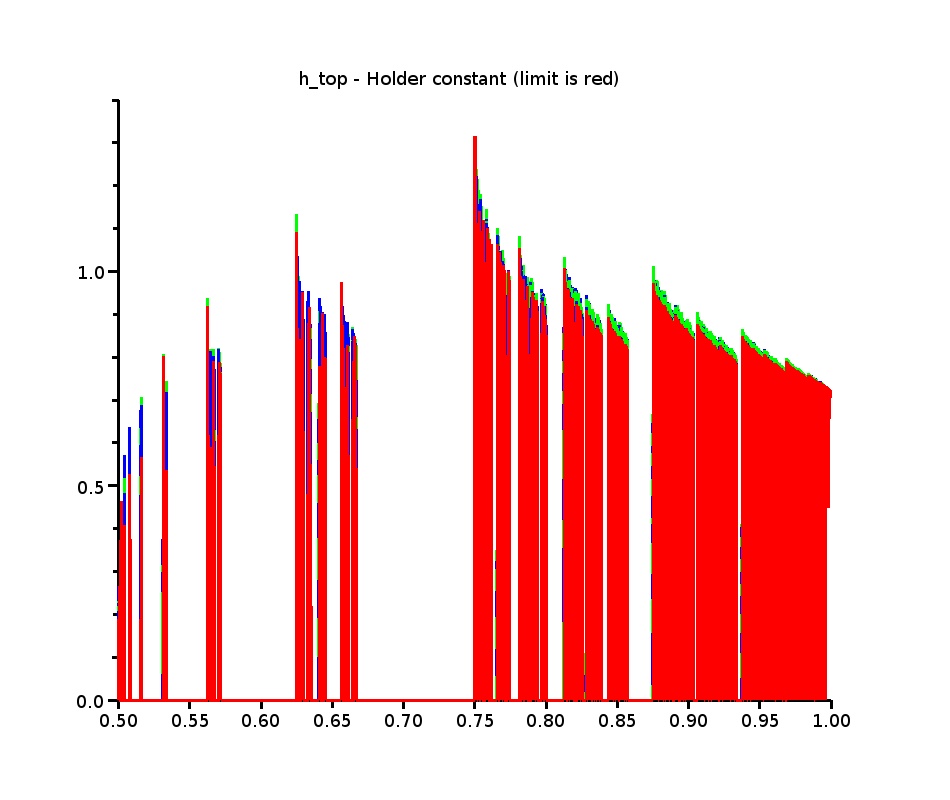}
\caption{Left: The topological entropy of the doubling map with hole $[a,1]$ as a function
  of $a$. 
  Right: Estimated local H\"older constant $C(a)$ as a function of $a$
  assuming $p=1$ using different mesh-sizes.}
\label{fig:1}
 \end{figure}

In this section we will illustrate and compare
H\"older continuity of families of interval
maps numerically and theoretically. 
A non-trivial difficulty in applying our theorems
numerically is that an explicit value for the
H\"older exponent  requires 
knowledge of the order $p$ of the pole 
associated with the leading eigenvalue. 
In our first example below we exhibit a family 
in which the leading  eigenvalue is simple for all 
parameter values, which in particular implies that $p=1$.
We then provide an example exhibiting
a double pole and show how this is reflected in  numerical experiments.

\subsection{Doubling map with left expanding hole}
We consider the map $T(x)=2 x \mod 1$,
 defined on $(0,1/2) \cup (1/2,1)$.
Let $1/2 <a<1$. We will be interested in the topological entropy
of $h(a)=\htop(T,[a,1])$, that is, $T$ restricted to the set $A=(0,1/2) \cup (1/2,a)$.
As shown in the previous section we may
calculate this entropy as the leading eigenvalue
of a transfer operator.  There are essentially two
ways to do so.
The first is to consider the map
$T$ as defined on  $I= (0,1/2)\cup (1/2,1)$, and using as weight $g$
the indicator function of $A$. The transfer operator then takes the form 
\begin{equation*}
  (\tr_a f) (x) = f \left( \frac{x}{2} \right) 
     + \chi_{(0,a)} \left(\frac{1+x}{2}\right) f\left(\frac{1+x}{2}\right)\,.
\end{equation*}
We note
that the map has big images as 
$T^n(Z)=(0,1)$ for every $n\in \N$, $Z\in\Zp_n$. We also have
$\Lambda(T)=\Xi(T)=\log 2$. 
We may apply Theorem~\ref{theorem:Main}
to see that $h$ is 
H\"older continuous at every $a\in (1/2,1)$ with the local H\"older exponent satisfying 
   \begin{equation}
     \Hol(h,a) \geq  \frac{h(a)}{p \log 2}
     \end{equation}
for some $p\in \N$.
        We claim that $p=1$ for each $a\in (1/2,1)$. This is also indicated
        by numerical experiments illustrated in Figure~\ref{fig:1}.
In that figure we show our numerical estimates of 
$h(a)=\htop(T,[a,1])$ on the left,  and on the right the local 
H\"older constant under the assumption that $p=1$ given by 
   \[ C(a)= \limsup_{s\to a} \frac{|h(s)-h(a)|}{(s-a)^{h(a)/\log 2}} \]
The (numerically obtained) limit superior seems to be neither 0 nor infinity,
indicating that our estimate for the local H\"older exponent is {\em optimal}.

Now, in order to prove that $p=1$ 
        it suffices to show that
 the leading eigenvalue is simple.
We shall do so by considering an alternative description of the system,
using a Hofbauer tower (\cite{H}, see also \cite[Section~3]{BK}), 
which we also used to compute our numerical estimates.

We start out by writing $J_0=(0,1/2)$, $J_1=(1/2,a)$, and $J_2=(a,1)$. 
Set ${\mathcal Z} = \{J_0,J_1,J_2\}$ and consider the locally constant weight 
$g=\chi_{J_1\cup J_2}$. 
For simplicity of notation 
we shall in the following tacitly omit intervals arising 
from $J_2$ from the Hofbauer tower
as they will not contribute to the spectral properties of the transfer operator on the tower. 

Instead of working with cylinder sets, that is, intersections of
preimages of the $J_i$'s, the idea is to deal 
with intersections of forward iterates of these intervals.
In our case, this boils down to studying 
the forward orbit of the point $a$.
To this end, we define the sequence $(a_n)_{n\geq 0}$
by setting $a_0=a$ and
 $a_{k+1}=T (\min \{ a, a_k \}^-) \in [0,  a_1]$ for $k\geq 0$, where
 $T(x^-)=\lim_{\xi \uparrow x}T(\xi)$ for $x\in(0,1]$.

The sequence contains finitely many values
if either 
$a_k\geq a$
for some $k$, which happens when
the orbit of $a$ `escapes', 
or if  $a$ is pre-periodic for $T$. Otherwise the orbit is infinite.

Let us first consider the case when
 the orbit is infinite, 
so in particular we have $0<a_k<a$ and $a_k\neq 1/2$
for all $k\geq 1$.
We set $\calD_0=T(J_0)=(0,1)$ and define
the `flats' $\calD_k=(0,a_k)$ for $k\geq 1$.
We have the following `transition' rules giving
rise to a transition matrix $M$ (elements other than those
mentioned all being zero):
\[
   \begin{array} {lcllll}
      \   0 < a_k < 1/2  & :  &
                 \calD_{k+1} &= T(\calD_k \cap J_0) &= (0, a_{k+1}), 
                    & M_{k,k+1}=1,\\
      \  1/2 < a_k < 1  & :  &
                 \calD_{k+1} &= T(\calD_k \cap J_1) &= (0, a_{k+1}),
                   & M_{k,k+1}=1, \\
      \  1/2 \leq  a_k < 1  & :  &
                 \calD_0 &= T(\calD_k \cap J_0) &= (0, 1),
                    & M_{k,0}=1.\\
   \end{array}
 \]
Note that the above lines describe two cases: 
if $a_k < 1/2$, then there is exactly one transition, namely from 
$\calD_k$ to $\calD_{k+1}$; 
if $a_k > 1/2$, then two different transitions are possible, namely the one from 
$\calD_k$ to $\calD_{k+1}$ and an additional one from $\calD_k$ to the bottom flat $\calD_0$.

For later use, let $A= \all{k \geq 0}{a_k \geq 1/2}$ denote the set of indices for which
there is a transition from the $k$-th flat to the bottom flat. 
Note that $0\in A$ and that $A$ contains at least one more index, the latter being a consequence of the fact 
that $T$ is expanding. 

The Hofbauer tower is now defined as the disjoint union 
of flats 
\[ \hatcalD = \all{ (x,\calD_k)}{k \geq 0,\  x\in \calD_k}\,.\]  
With the tower we associate a directed graph $\calG$
obtained
from the transition matrix.
A possible transition graph is sketched in Figure~\ref{fig:graph}.

\begin{figure}
\includegraphics[width=0.5\paperwidth]{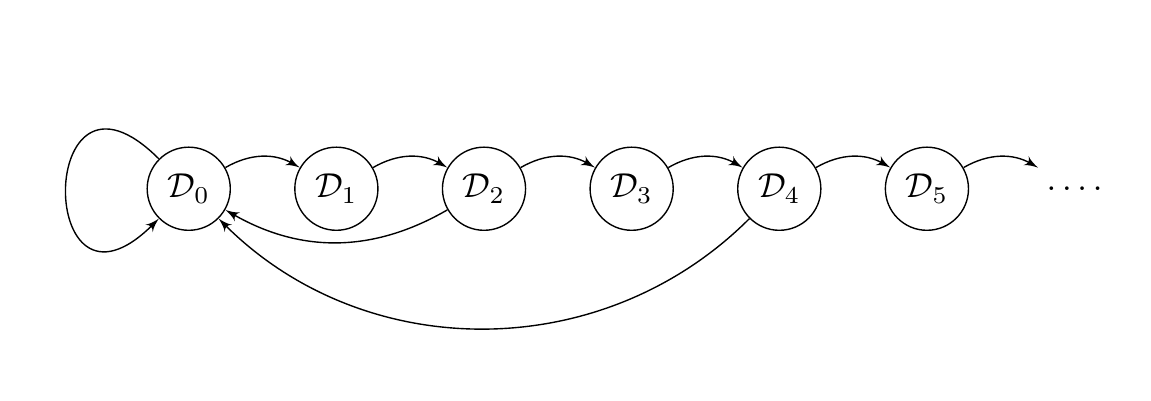}
\caption{The transition graph $\calG$ in the case of $A=\{0,2,4,\ldots\}$}
\label{fig:graph}
 \end{figure}

By \cite[Theorem 2 and Lemma 4.1]{BK}, the multiplicity
of the leading eigenvalue $\lambda$ of the transfer operator $\tr_a$ 
equals the order of the leading pole $z=1/\lambda$
of the zeta-function associated
with the tower, which is given by 
 \[ \zeta(z) = \exp \sum_{k=1} ^\infty \widehat{\zeta}_k \frac{z^k}{k}\,, \]
where $\widehat{\zeta}_k$ is the total number of periodic
points in 
$\calG$ of period $k$ (not necessarily prime).

Now, by Hofbauer
\cite[Theorem 1 as well as Lemmas 3 and 4]{H} this 
periodic orbit counting zeta-function 
has the same 
poles in the unit disk as the
zeros of a determinant calculated in the following way: 
call $\gamma=(i_0 i_1 \cdots i_{n-1})$ a simple cycle if
$M_{i_0 i_1} \cdots M_{i_{n-1}i_0}=1$ and all the $i$'s are distinct.
Also denote by $|\gamma|=n$ the length of such a cycle.
Then 
\begin{equation}
    d(z) = 1 +  \sum 
      \left( - z^{|\gamma_1|} \right)  \cdots
      \left( - z^{|\gamma_q|} \right)  
      \label{eq:hofbauer}
\end{equation}
where the sum is over all 
$q$-tuples (with $q\geq 1$)
of disjoint simple cycles of $\calG$.

When a transition matrix $M$ is of finite size $N\times N$
the formula follows easily from
the standard expansion of
$d(z)=\det(I - z M)$ in terms of permutations
and rewriting permutations as products
over distinct cycles.
In the case of unbounded matrix size
we refer to \cite{H} which explains how to take a limit
of finite matrix truncations, that is, levels in the Hofbauer tower.
Combining the above two results we have: 

\begin{theorem} 
Let $z\in \C$ with $0<\abs{z}<1$. Then $z$ is a zero of $d$ iff
$1/z$ is an eigenvalue of the transfer operator. Moreover, if this is the case the order of the zero
equals the (algebraic) multiplicity of the eigenvalue.
\label{thm:hofbauer}
\end{theorem}

In order to apply the theorem above to the present situation we requite a simple lemma on the nature of the zeros closest to the origin of certain power series. 

\begin{lemma}
\label{lem:dlem}
Let $(m_k)_{k\geq 1}$ denote a sequence with $m_k\in \set{0,1}$ such that 
 $m_k=1$ for at least one $k\geq 1$. Then 
\[ \delta(z)=1-z-\sum_{k=1}^\infty m_kz^{k+1}\]
is holomorphic in the open unit disk. Moreover,
there is a unique $r\in (0,1)$ which is a simple zero of $\delta$ and all other zeros of $\delta$ are strictly 
larger in modulus. 
 \end{lemma}
\begin{proof}
By the Cauchy-Hadamard Theorem, the function $\delta$ is holomorphic in the open unit disk. Furthermore, 
we have $|\delta(z)| \geq 1 - \abs{z} - \sum_{k=1}^\infty m_k |z |^{k+1} =
\delta(|z|)$ with equality iff $z$ is real and positive.
Since $\delta(0)=1$ and $\delta(z)<0$ for $z$ real and sufficiently close to $1$ there is  
$r\in (0,1)$ with $d(r)=0$. 
Moreover, since 
$|\delta(\lambda)|>\delta(|\lambda|)$ when $\lambda \notin [0,1)$,
any other zero must be of absolute value strictly larger
than $r$, and noting that $d'(r)=-1-\sum_{k=1}^\infty (k+1) m_k r^k <0$ the order of the
zero must be one. 
\end{proof}

Returning to our specific case when the orbit of $a$ is infinite
we obtain the following. 

\begin{proposition}
   Suppose that the orbit $(a_k)_{k\geq 0}$ is
   infinite. Then the Hofbauer determinant
      (\ref{eq:hofbauer})
   is given by 
  \begin{equation*}
     d(z) = 1 -  \sum_{k=0}^\infty M_{k,0} \, z^{k+1}
  \end{equation*}
  and is holomorphic in the open unit disk. 
  There is a unique $r\in (0,1)$ which is a simple zero of $d$
  and all other zeros of $d$ are of strictly larger modulus.
  The spectral radius of the transfer operator equals $1/r$ and
  is a simple eigenvalue. 
 All other spectral values are strictly smaller in modulus than $1/r$.
\end{proposition}

\begin{proof}
Cycles are to be distinct in the above sum and
{\em any} cycle in our transition graph has to contain $\calD_0$ so 
there are no two distinct non-trivial cycles.
Since every cycle
is of the form $(\calD_0\calD_1 \cdots \calD_{k})$  with $M_{k,0}=1$
we immediately obtain the expression for the Hofbauer determinant.

The remaining assertions now follow from Theorem~\ref{thm:hofbauer} and Lemma~\ref{lem:dlem} by 
observing that $d(z)=1-\sum_{k\in A}^\infty z^{k+1}$, that $0\in A$ and that $A\setminus \set{0}$ is 
non-empty. 
\end{proof}

We now turn to the case where the Hofbauer tower is finite. Recall that this happens precisely 
if $a$ is pre-periodic for $T$ or if the orbit escapes. 
In either case there are $N$ and $1\leq j \leq N$ with $a_{N+1}=a_j$;  note that if the orbit escapes 
we have $a_{N+1}=a_1$.  
Thus there is no element $M_{N,N+1}$ but we have the additional transition
$M_{N,j}=1$.

We then have the following simple cycles:

\begin{enumerate}
\item For $0\leq k\leq N$ with $M_{k,0}=1$, 
   $(\calD_0 \calD_1 \cdots \calD_k)$ is of length $k+1$.
\item One additional cycle of the form 
   $(\calD_j  \cdots \calD_N)$ and of length $N-j+1$.
\end{enumerate}

In this case there is also the possibility of disjoint $2$-tuples
for each $0\leq k <j$ with $M_{0,k}=1$ of the form
  \begin{equation}
   \{(\calD_0 \calD_1 \cdots \calD_k),
   (\calD_j  \cdots \calD_N) \}.
  \end{equation}

The Hofbauer determinant
      (\ref{eq:hofbauer})
then takes the form
\begin{align*}
  d(z) =& 1 - 
        z^{N-j+1}  -
        \sum_{k=0}^N M_{k,0} \, z^{k+1} 
            + \sum_{k=0}^{j-1}  M_{k,0} \, z^{k+1} z^{N-j+1} \nonumber \\
     =& 
       \left(1 - z^{N-j+1}\right)  
        \left( 1 - 
             \sum_{k=0}^{j-1} M_{k,0} \,  z^{k+1} - 
              \left(1 - z^{N-j+1}\right)^{-1} 
             \sum_{k=j}^{N} M_{k,0} \,  z^{k+1}  \right) \nonumber \\
     =& 
       \left(1 - z^{N-j+1}\right)  
     \left( 1 -  \sum_{k=0}^{\infty}
        \widetilde{M}_{k,0} \,  z^{k+1} \right)
 \nonumber \\
\end{align*}
where 
\begin{equation}
\label{eq:tildeM}
\widetilde{M}_{k,0} = 
 \begin{cases} 
    M_{k,0} & \text{for $0\leq k<j$,}\\
    M_{j+(k-j \mod (N-j+1)),0} & \text{for $k\geq j$.} 
    \end{cases}
\end{equation}
We may thus conclude:
\begin{proposition}
   Suppose that the orbit $(a_k)_{k\geq 0}$ is
   finite. Then the Hofbauer determinant
      (\ref{eq:hofbauer})
   is given by 
  \begin{equation*}
    d(z) = 
     \left(1 - z^{N-j+1}\right)  
     \left( 1 -  \sum_{k=0}^{\infty}
        \widetilde{M}_{k,0} z^{k+1} \right )
  \end{equation*}
for some $N$ and $j$ with $1\leq j \leq N$ and  $\widetilde{M}_{k,0}$ gven by (\ref{eq:tildeM}). 
The function $d$ is holomorphic in the open unit disk. 
  Moreover, there is a unique $r\in (0,1)$ which is a simple zero of $d$
  and all other zeros of $d$ are of strictly larger modulus.
  The spectral radius of the transfer operator equals $1/r$ and
  is a simple eigenvalue. 
 All other spectral values are strictly smaller in modulus than $1/r$.
\end{proposition}
\begin{proof}
This follows from the calculation above together with Theorem~\ref{thm:hofbauer}
and Lemma~\ref{lem:dlem}. 
\end{proof}
Summarising, we have shown the following. 
\begin{corollary}
For any $a\in (1/2,1)$ the transfer operator $\tr_a$ has a leading eigenvalue which is simple. In particular, the topological entropy $h(a)$ of the doubling map with hole $[a,1]$ satisfies 
\[   \Hol(h,a) \geq  \frac{h(a)}{\log 2}\,.\]
\end{corollary}

\subsection{Doubling map with hole giving rise to a double pole}
One of the dynamically simplest examples exhibiting a double pole
of the resolvent is again furnished by the doubling map 
$T(x)=2x \mod 1$. This time, however, we introduce
a hole from $a$ to $a+1/12$, that is, we shall consider 
$\htop(T,[a,a+1/12])$ as a function of $a$. 
For the particular
value $a=3/4$ the hole is 
from $3/4$ to $5/6$.
Figure~\ref{fig:2} shows this map indicating orbits of boundary
points. 

\begin{figure}
 \includegraphics[width=0.3\paperwidth, bb=20 20 575 575 ]{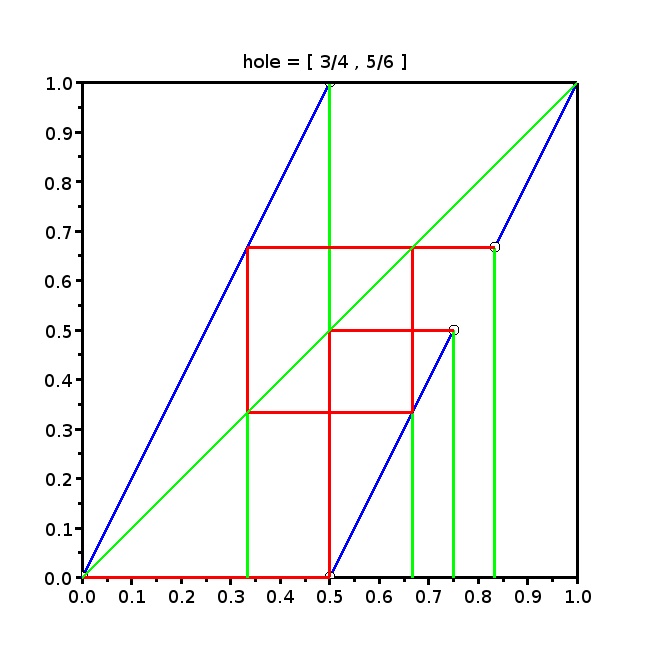} 
\caption{The doubling map with hole $[3/4,5/6]$ showing orbits of the hole boundaries.}
 \label{fig:2}
 \end{figure}

The estimated entropy $\htop(T,[a,a+1/12])$ 
as a function of $a$ is shown in Figure~\ref{fig:3} on the left, exhibiting a `dip' at $a=3/4$. 
In order to understand this dip, note that
the intervals $I_1=(0,1/2)$, $I_2=(1/2,3/4)$, and $I_3=(5/6,1)$ make up 
the initial partition. 
Introducing $I_4=(2/3,3/4)$ and $I_5=(1/3,1/2)$
we have the following transitions
(see Figure \ref{fig:double} for the corresponding transition graph):
\[ \ncircle{1} \to \text{\ncircle{1}, \ncircle{2} or \ncircle{3}}\,; 
   \quad \ncircle{2} \to \ncircle{1}\,; 
   \quad \ncircle{3} \to \text{\ncircle{3} or \ncircle{4}}\,; 
   \quad \ncircle{4} \to \ncircle{5}\,; 
   \quad \ncircle{5} \to \text{\ncircle{3} or \ncircle{4}}\,.\]

\begin{figure}[h]
\includegraphics[width=0.5\paperwidth]{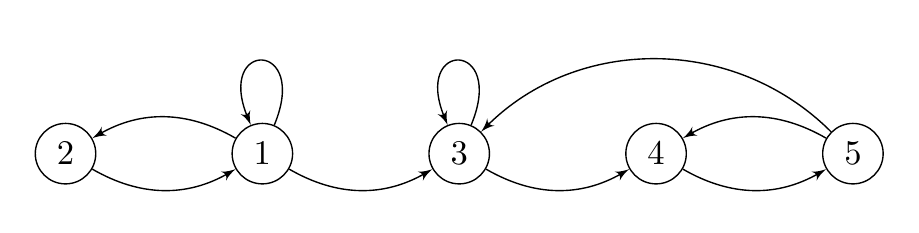}
\caption{The transition graph $\calG$ in the double pole case}
\label{fig:double}
 \end{figure}

One verifies  that $\gamma=(1+\sqrt{5})/2$ is
a double eigenvalue of the corresponding transition
matrix $\pi$ (see Figure~\ref{fig:3} right)
and that 
$\ker(\pi - \gamma I)$ is one-dimensional.
This implies that the resolvent has a double pole at $\gamma$.
 Theorem  \ref{theorem:Main} applied at
 $a=3/4$ shows that the entropy $h(a)=\htop(T, [a,a+1/12])$
 is H\"older continuous with a local exponent 
 at least $ \log(\gamma)/(2 \log 2) \approx 0.3471$ which is
 consistent with a numerical estimate for the exponent (not shown).

 \begin{figure}[h]
 \vspace{4em}
\hspace{-5em}
\includegraphics[width=0.15\paperwidth, bb=20 20 575 575]{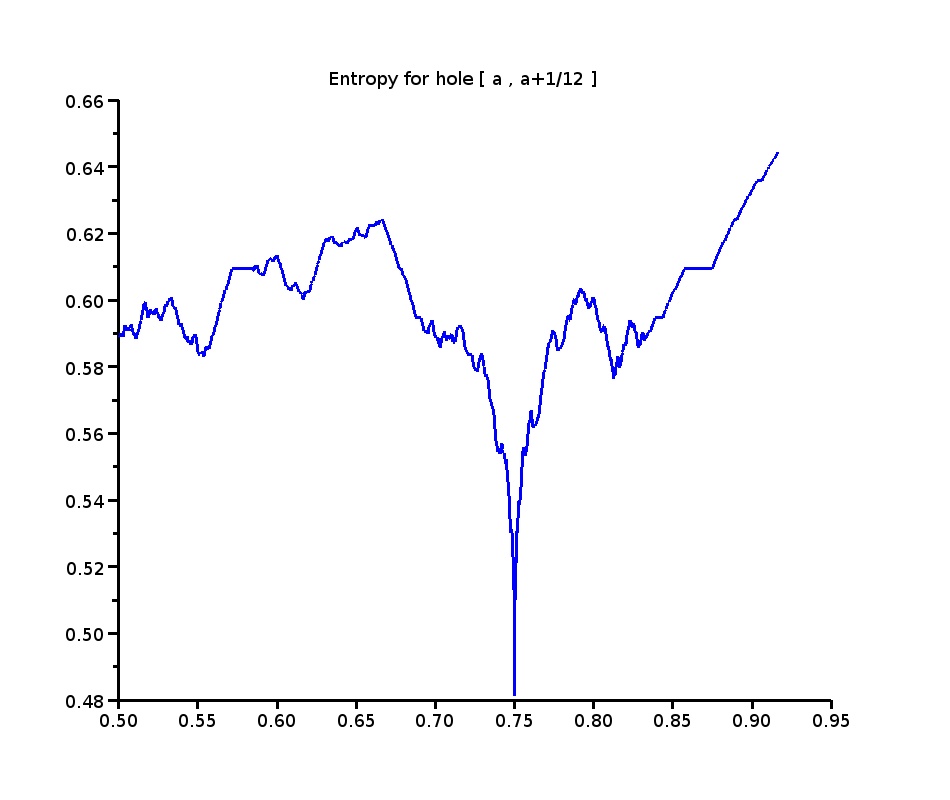} \hspace*{2cm}
\includegraphics[width=0.15\paperwidth, bb=20 20 575 575]{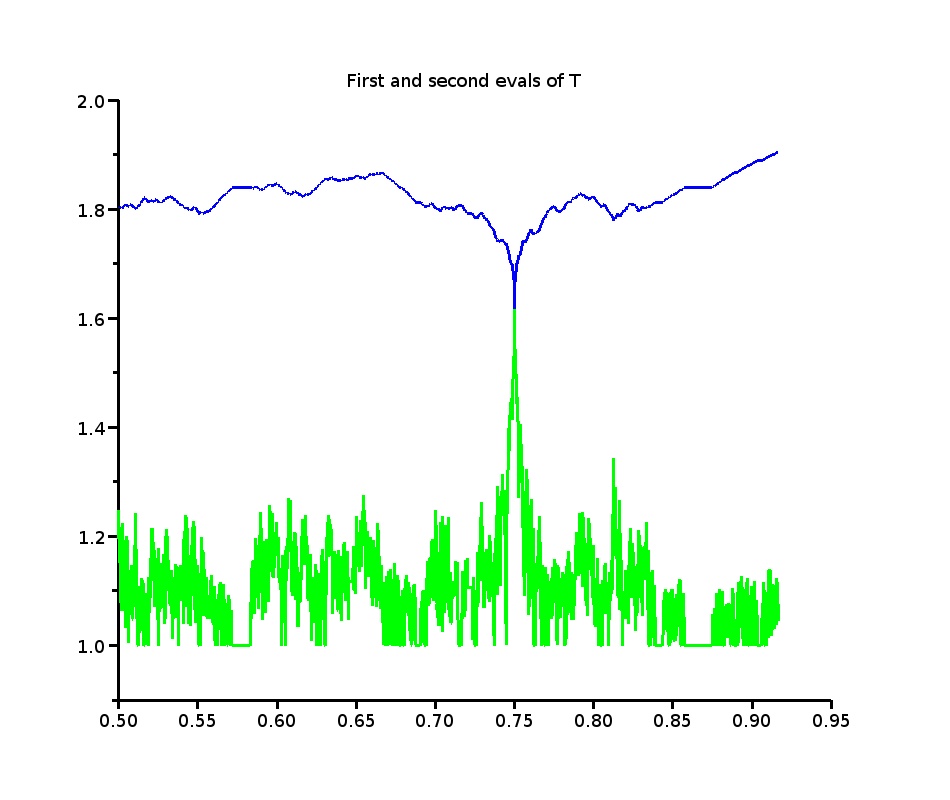} \ \ \ 
\caption{Left: the topological $\htop(T,[a,a+1/12])$ as a function of $a$. 
     Note the dip for $a=3/4$.
  Right: the two largest eigenvalues of the transfer operator.}
\label{fig:3}
 \end{figure}

\section*{Acknowledgements}
The research in this article was carried out while the first author was visiting the D\'epartement de 
Math\'ematiques d' at the Universit\'e Paris-Sud during research leave from Queen Mary, University of
London. Both authors are grateful to Viviane Baladi, Gerhard Keller and Henk Bruin for valuable feedback 
during the preparation of this article, as well as to an anonymous referee, whose comments led to a 
considerable simplification of the arguments in Section~\ref{sec:ex}.

\end{document}